\documentclass{amsart}

\usepackage{amsmath,mathrsfs}
\usepackage{amssymb}
\usepackage{cite}
\usepackage{graphicx, epsfig}
\usepackage{hyperref}
\usepackage{soul}
\usepackage{color}

\usepackage[margin=3cm]{geometry}
\hypersetup{
 colorlinks   = true,
 urlcolor     = blue,
 citecolor    = red,
 bookmarksopen = true
}

\newtheorem{theorem}{Theorem}[section]
\theoremstyle{plain}

\newtheorem{corollary}[theorem]{Corollary}

\newtheorem{definition}[theorem]{Definition}

\newtheorem{lemma}[theorem]{Lemma}

\newtheorem{proposition}[theorem]{Proposition}
\newtheorem{remark}[theorem]{Remark}

\numberwithin{equation}{section}
\title[]{Exceptional Boundary Sets for Solutions of Fully Nonlinear Parabolic PDEs}
\author{Ram Baran Verma}
\address[Ram Baran Verma]{Department of Mathematics, SRM University Amaravati, Andhra Pradesh-522502, India}
\email{rambaran.v@srmap.edu.in}
\author{Mohan Mallick}
\address[Mohan Mallick]{Visvesvaraya National Institute Of Technology, India-522502}
\email{mohan.math09@gmail.com, mohanmallick@mth.vnit.ac.in}

\date{}
\begin{document}

\begin{abstract}
This article investigates the exceptional set of the boundary for the following problem:
\begin{equation*}\label{excep}
\begin{aligned}
-\frac{\partial u}{\partial t} + \mathcal{M}_{\lambda,\Lambda}^+(D^2u) + b(x,t)\cdot Du + c(x,t)u =0 \quad \rm{in} ~ \Omega_{T},
\end{aligned}
\end{equation*}
We provide a sufficient condition on the exceptional set in terms of the bound of the Hausdorff measure of this boundary portion. This condition ensures that even if the boundary values are not nonnegative on this portion, the supersolution remains nonnegative.
\end{abstract}

\maketitle

\section{Introduction}

In this article, we consider the following parabolic problem:
\begin{equation}\label{excep}
\begin{aligned}
-\frac{\partial u}{\partial t} + \mathcal{M}_{\lambda,\Lambda}^+(D^2u) + b(x,t) \cdot Du + c(x,t)u = 0 \quad \text{in} \ \Omega_{T},
\end{aligned}
\end{equation}
where $\Omega_{T} = \Omega \times (0, T)$ and $\Omega \subset \mathbb{R}^{n}$. The function $b(x,t): \Omega_{T} \rightarrow \mathbb{R}^{n}$ is a bounded vector field, and $c$ is a bounded function. The precise conditions on $b$ and $c$ will be given below (see $(A1)$). $\mathcal{M}_{\lambda,\Lambda}^+$ is Pucci's extremal operator defined by \eqref{pucci}.

The term "exceptional sets" appears in various contexts and generally refers to situations where certain properties are retained if we ignore the "exceptional" set. One of the problems related to exceptional sets is measuring their "size." In the context of partial differential equations or related areas, size usually means a certain measure (Lebesgue, Hausdorff) or, more generally, capacity. For an introduction to terms like Hausdorff measure and capacity, we refer to \cite{evans2018measure, heinonen2018nonlinear}. A more detailed analysis of their relation to exceptional sets and related problems can be found in \cite{carleson1967selected}. For the exceptional set of a Sobolev function and its relation to capacity, see Section 2.42 in \cite{heinonen2018nonlinear}.

This article deals with the uniqueness of the solution to \eqref{excep}. The parabolic boundary $\Omega_{T}$ consists of two parts: the base and the lateral portion of the boundary. We will provide sufficient conditions for the exceptional sets of the base and lateral parts of the boundary separately. As discussed above, an exceptional set in this context is a subset of the boundary such that if we change the boundary values on these sets, the corresponding solution remains unchanged. For divergence form operators, the exceptional set of the boundary and the corresponding uniqueness has been discussed in \cite{gauidenko1981exceptional} and Chapter 6 of \cite{chabrowski2006dirichlet}. For linear elliptic and parabolic equations of the non-divergence form, we refer to \cite{hile1986exceptional} and \cite{hile1988exceptional}, respectively. Here, we consider the fully nonlinear operator of the non-divergence form.

This type of problem is closely related to the results of Zaremba and Phragmén-Lindelöf, which state that if $u$ is subharmonic in a domain, even though 
\[\lim_{y \rightarrow x} \inf_{y \in \Omega} u(y) \geq 0\]
fails to hold at some (single) point of the boundary, the solution is still non-negative in the whole domain. For results related to Zaremba, we refer to Theorem 3.6.29 in \cite{helms2009potential}. It is noteworthy that the Phragmén-Lindelöf theorem was first established by D. Gilbarg and E. Hopf for the Laplace equation in two and arbitrary finite dimensions, respectively (see \cite{gilbarg1952phragmen, hopf1952remarks}). We also refer to \cite{oddson1968some, oddson1969phragmen}. Later, Miller obtained the Phragmén-Lindelöf theorem for general uniformly elliptic equations without the continuity assumption on the coefficients in \cite{miller1967barriers}. For a nice introduction and development of this topic, we refer to \cite{protter2012maximum}. In the context of fully nonlinear elliptic equations, this result was obtained by I. C. Dolcetta and A. Vitolo \cite{dolcetta2007qualitative}. This result has been extended to the case of unbounded coefficients in \cite{punzo2010phragmen}. Recently, this result has also been extended to fully nonlinear parabolic equations with unbounded measurable coefficients \cite{tateyama2020phragmen}.

Here, our interest is to extend the result of Hile and Yeh \cite{hile1988exceptional} in the context of fully nonlinear parabolic equations, which concerns not only a single point but a subset of the boundary. We have obtained a sufficient condition in terms of the Hausdorff measure of the exceptional set, so that if the boundary values are changed on this subset, the solution remains unchanged. The Hausdorff dimension in this case will be $(n-1)$. We will see that for the base boundary, the upper bound of the Hausdorff measure depends only on the ellipticity constant, while for the lateral boundary, it depends on the size of the singular barrier constructed by Miller. For more details, see Section 3.

This article is organized as follows. In Section 2, we have collected some elementary definitions and results that will be used throughout this article. Section 3 is devoted to the proof of the main result in the case of the base boundary. In the last section, we construct the singular barrier following the approach of Miller and prove the main result of the article in the case of the lateral boundary.
\section{Auxiliary Notation and Definition}

In this section, we collect definitions and results that will be used throughout the article. For given constants $0 < \lambda < \Lambda$, Pucci's extremal operators are defined as follows:
\begin{equation}\label{pucci}
\mathcal{M}^{\pm}_{\lambda,\Lambda}(M) = \Lambda \sum_{\pm e_{i} > 0} e_{i} + \lambda \sum_{\pm e_{i} < 0} e_{i},
\end{equation}
where $M$ is a symmetric matrix of size $N \times N$ and $e_i$ are its eigenvalues. In the definition of Pucci's extremal operator, we need to compute the eigenvalues of the Hessian $D^{2}u$ of the function. In general, it is very difficult to compute the eigenvalues of the Hessian. However, if the function $u$ is radial, then the eigenvalues of $D^{2}u$ are given by the following lemma.

\begin{lemma}[Lemma 3.1 \cite{AF1}]\label{radial}
Let $\tilde{u}: [0, \infty) \rightarrow \mathbb{R}$ be a $C^{2}$ function such that $u(x) = \tilde{u}(|x|)$. Then for any $x \in \mathbb{R}^{N} \setminus \{0\}$, the eigenvalues of the Hessian $D^{2}u(x)$ are $\frac{\tilde{u}'(|x|)}{|x|}$ with multiplicity $N-1$ and $\tilde{u}''(|x|)$ with multiplicity 1.
\end{lemma}

Although the definition of the Hausdorff measure and Hausdorff dimension of a set is well known, we recall it here for the sake of completeness.

\begin{definition}
Let $A \subset \mathbb{R}^{n}$ and $0 \leq s < \infty$. We define the $s$-dimensional Hausdorff measure to be:
\[
\mathscr{H}^{s}(A) := \lim_{\delta \rightarrow 0} \mathscr{H}_{\delta}^{s}(A) = \sup_{\delta > 0} \mathscr{H}_{\delta}^{s},
\]
where
\[
\mathscr{H}_{\delta}^{s}(A) := \inf \left\{\sum_{i=1}^{\infty} \alpha(s) \left(\frac{\text{Diam} \ C_{i}}{2}\right)^{s} \mid A \subset \bigcup_{i=1}^{\infty} C_{i}, \ \text{Diam}(C_{i}) \leq \delta \right\},
\]
and
\[
\alpha(s) = \frac{\pi^{\frac{s}{2}}}{\Gamma(\frac{s}{2}+1)}.
\]
\end{definition}

We also recall a simple corollary which will be of use.

\begin{lemma}\label{house}
Let $E$ be a subset of $\mathbb{R}^n$ with Hausdorff dimension $\mathscr{H}(E)$. If $\mu$ is a positive number such that $\mathscr{H}(E) < \mu$, then for any $\varepsilon > 0$ and $\nu > 0$, there exists a countable covering of $E$ by open balls $\{B_i\}$ in $\mathbb{R}^n$ such that the center of each $B_i$ is a point of $E$, each radius $r_i$ of $B_i$ satisfies $0 < r_i \leq \nu$, and moreover, $\sum_i r_i^\mu < \varepsilon$.
\end{lemma}

\begin{lemma}[Comparison Principle]
Consider the equation
\begin{equation}
-\frac{\partial u}{\partial t} + \mathcal{M}_{\lambda,\Lambda}^+(D^2u) + b(x,t) \cdot Du + c(x,t)u = f(x,t) \quad \text{in} \ \Omega_{T}.
\end{equation}
Suppose that $u$ is a subsolution and $v$ is a supersolution such that $u \leq v$ on $\partial_{\mathscr{P}} \Omega_{T}$. Then, $u \leq v$ in $\Omega_{T}$.
\end{lemma}

\begin{remark}
Through this article we will be assuming $c\leq 0.$
\end{remark}

\begin{remark}
If $(x_{0}, t_{0})$ is a point on $\partial_{P} \Omega_{T}$ and $u$ is a real-valued function in $\Omega_{T}$, by the phrase $u \geq 0$ at $(x_{0}, t_{0})$, we mean that $\lim \inf u(x, t) \geq 0$ as $(x, t)$ approaches $(x_{0}, t_{0})$ from inside $\Omega_{T}$. If $\mathscr{D}$ is a subset of $\partial \Omega_{T}$, by $u \geq 0$ on $\mathscr{D}$ we mean that $u \geq 0$ at every point $(x, t)$ on $\mathscr{D}$. If $u \geq 0$ on the intersection of $\partial \Omega_{T} - \mathscr{D}$ with a neighborhood of $(x, t)$ (or of $\mathscr{D}$), then we say $u \geq 0$ on $\partial \mathscr{D} - E$ near $(x, t)$ (or, respectively, near $E$).
\end{remark}
\section{Main Result in the Case of the Base Boundary}

In this section, we prove Theorem \ref{bmain}, which asserts a sufficient condition in terms of the Hausdorff measure of the exceptional set. To prove this theorem, we need several elementary results as discussed in \cite{hile1988exceptional}.

Consider the problem:
\begin{equation}\label{excep}
\begin{aligned}
-\frac{\partial u}{\partial t} + \mathcal{M}_{\lambda,\Lambda}^+(D^2u) + b(x,t) \cdot Du + c(x,t)u = 0 \quad \text{in} \ \Omega_{T},
\end{aligned}
\end{equation}
where $\Omega_{T} = \Omega \times (0,T)$, and $b(x,t) = (b_1(x,t), b_2(x,t), \dots, b_n(x,t))$, and $c(x,t)$ are functions defined on $\Omega_{T}$. The basic assumptions on the coefficients $b$ and $c$ are as follows:
\begin{enumerate}
\item[(A1)] There are two non-negative functions $b_0$ and $c_0$ defined on $(0,T)$ such that for all $(x,t) \in \Omega_{T}$, we have
\[|b(x,t)| \leq b_{0}(t) = o(t^{-\frac{1}{2}}) \quad \text{and} \quad |c(x,t)| \leq c_{0}(t) = o(t^{1-\beta}),\]
for some $0 < \beta < 1$. Observe that this condition also implies $c_{0}(t) = o(t^{-1})$ as $t \rightarrow 0$.
\end{enumerate}

\begin{lemma}
Given two positive constants $\alpha$ and $\sigma$ satisfying
\begin{equation}\label{ineqa}
0 < 2\alpha < 4n\lambda\sigma < \frac{\lambda}{\Lambda},
\end{equation}
define 
\[\psi(x,t) = t^{-\alpha} e^{-\frac{\sigma |x|^{2}}{t}}.\]
Then there exist positive constants $T_1$ and $\gamma_1$ depending on the ellipticity constants, $b_0$, $c_0$, $\sigma$, and $\alpha$, such that for $(x,t) \in \Omega_{T_{1}}$, we have
\begin{equation}\label{req1}
-\frac{\partial \psi}{\partial t} + \mathcal{M}_{\lambda,\Lambda}^+(D^2\psi) + b(x,t) \cdot D\psi + c(x,t)\psi \leq -\gamma_1 (t + |x|^{2}) t^{-2} \psi(x,t),
\end{equation}
and
\begin{equation}\label{req11}
\left\{
\begin{aligned}
(i) \quad |D_{i}\psi(x,t)| &\leq (\gamma_{1}t)^{-1}|x|\psi(x,t) \quad \text{for} \ i = 1, 2, \ldots, n, \\
(ii) \quad |D_{ij}\psi(x,t)| &\geq \gamma^{-1}_{1} t^{-2} (\delta_{ij}t + |x|^{2}) \psi(x,t) \quad \text{for} \ i, j = 1, 2, \ldots, n, \\
(iii) \quad |D_{t}\psi(x,t)| &\leq \gamma^{-1}_{1} t^{-2} (\delta_{ij}t + |x|^{2}) \psi(x,t).
\end{aligned}
\right.
\end{equation}
\end{lemma}

\begin{proof}
From the expression of $\psi$, it is clear that  
\begin{equation*}
\begin{aligned}
D\psi(x,t) &= \frac{-2\sigma}{t} x \psi, \\
\frac{\partial \psi}{\partial t} &= \left(-\frac{\alpha}{t} + \frac{\sigma |x|^{2}}{t^{2}}\right) \psi, \\
D^{2}\psi(x,t) &= \left(\frac{-2\sigma}{t} I + \frac{4\sigma^{2}}{t^{2}} x \otimes x\right) \psi.
\end{aligned}
\end{equation*}

Consider 
\begin{align}
&-\frac{\partial \psi}{\partial t} + \mathcal{M}_{\lambda,\Lambda}^+(D^2\psi) + b(x,t) \cdot D\psi + c(x,t)\psi = \nonumber \\
&\left(\frac{\alpha}{t} - \frac{\sigma |x|^{2}}{t^{2}}\right) \psi + \psi \mathcal{M}_{\lambda,\Lambda}^{+} \left(\frac{-2\sigma}{t} I + \frac{4\sigma^{2}}{t^{2}} x \otimes x\right) - \frac{2\sigma}{t} \psi b(x,t) \cdot x + c(x,t)\psi \nonumber \\
&\leq \left(\frac{\alpha}{t} - \frac{\sigma |x|^{2}}{t^{2}}\right) \psi - \frac{2\sigma}{t} \psi \mathcal{M}_{\lambda,\Lambda}^{-}(I) + \frac{4\sigma^{2}}{t^{2}} \psi \mathcal{M}_{\lambda,\Lambda}^{+}(x \otimes x) + \frac{2\sigma}{t} b_{0}(t) |x| \psi + c_{0}(t)\psi \nonumber \\
&= \left(\frac{\alpha}{t} - \frac{\sigma |x|^{2}}{t^{2}}\right) \psi - \frac{2n\lambda\sigma}{t} \psi + \frac{4\sigma^{2}}{t^{2}} \Lambda |x|^{2} \psi + \frac{2\sigma}{t} b_{0}(t) |x| \psi + c_{0}(t)\psi. \label{above}
\end{align}

Now, using Young's inequality, we obtain:
\[2b_{0}(t)|x| \leq \epsilon \frac{|x|^{2}}{t} + \frac{t b_{0}^{2}(t)}{\epsilon},\]
where $\epsilon = (1 - 4n\sigma\Lambda)/2$. Observe that $\epsilon$ is a positive constant because of the condition \eqref{ineqa}. With this choice of $\epsilon$, the calculation in \eqref{above} takes the form:
\begin{align*}
&\leq \left(\alpha - 2\sigma\lambda n + \frac{2t\sigma b_{0}^{2}(t)}{1 - 4n\sigma\Lambda} + t c_{0}(t)\right) \frac{\psi}{t} + \left(-2 + 8\sigma\Lambda + 2\epsilon\right) \frac{\sigma |x|^{2}}{2t^{2}} \psi \\
&= \left(\alpha - 2\sigma\lambda n + \frac{2t\sigma b_{0}^{2}(t)}{1 - 4n\sigma\Lambda} + t c_{0}(t)\right) \frac{\psi}{t} - \left(1 + 8\sigma\Lambda \left(\frac{n}{2} - 1\right)\right) \frac{\sigma |x|^{2}}{2t^{2}} \psi.
\end{align*}

Now, we choose $T_{1}$ such that for all $t \in (0, T_{1})$ we have 
\[
\frac{2t\sigma b_{0}^{2}(t)}{1 - 4n\sigma\Lambda} + t c_{0}(t) \leq \frac{2\sigma\lambda n - \alpha}{2},
\]
which is always possible in view of the assumptions on $b_0$ and $c_0$. Then \eqref{req1} follows with \[\gamma_{1} = \min \left(\frac{2\sigma\lambda n - \alpha}{2}, \left(1 + 8\sigma\Lambda \left(\frac{n}{2} - 1\right)\right) \frac{\sigma}{2}\right).\]
\end{proof}

Now consider 
\begin{equation}\label{phi}
\phi(x,t) = t^{1-\beta} + (1 + t^{\beta}) |x|^{2} \quad \text{for} \ (x,t) \in \mathbb{R}^{n} \times (0, \infty),
\end{equation}
where $\beta$ is the same as in the definition of $c_{0}$.

\begin{proposition}\label{phi11}
There exist positive constants $T_{2}$ and $\gamma_{2}$ depending on $n, \Lambda, \beta, b_{0}$, and $c_{0}$ such that in $\Omega_{T_{2}}$ we have the following:
\begin{equation}\label{req222}
-\frac{\partial \phi}{\partial t} + \mathcal{M}_{\lambda,\Lambda}^+(D^2\phi) + b(x,t) \cdot D\phi + c(x,t)\phi \leq -\gamma_2 \left(t^{-\beta} + t^{\beta-1} |x|^{2}\right).
\end{equation}
It is also clear from the definition of $\phi$ that
\begin{equation}
\left|\frac{\partial \phi}{\partial t}\right| \leq \left[t^{-\beta} + t^{\beta-1} |x|^{2}\right].
\end{equation}
\end{proposition}

\begin{proof}
From the calculations, we find that 
\begin{equation*}
\begin{aligned}
D\phi(x,t) &= 2(1 + t^{\beta})x, \\
D^{2}\phi(x,t) &= 2(1 + t^{\beta})I.
\end{aligned}
\end{equation*}

\begin{align}
&-\frac{\partial \phi}{\partial t} + \mathcal{M}_{\lambda,\Lambda}^+(D^2\phi) + b(x,t) \cdot D\phi + c(x,t)\phi \nonumber \\
&\leq 2n\Lambda(1 + t^{\beta}) + 2(1 + t^{\beta}) b(x,t) \cdot x - (1 - \beta)t^{-\beta} - \beta t^{\beta-1} |x|^{2} + c(x,t) \left[t^{1-\beta} + (1 + t^{\beta}) |x|^{2}\right] \nonumber \\
&\leq 4n\Lambda + \frac{8}{\beta} t^{1-\beta} b_{0}^{2}(t) - (1 - \beta)t^{-\beta} - \beta t^{\beta-1} |x|^{2} + \{t c_{0}(t)\} t^{-\beta} + 2 \{t^{\beta-1} c_{0}(t)\} t^{1-\beta} |x|^{2}. \nonumber
\end{align}

In the above calculations, we have used the following Young's inequality:
\[b_{0}(t) \leq \frac{\beta}{8} t^{\beta-1} |x|^{2} + \frac{2}{\beta} t^{1-\beta} b_{0}^{2}(t).\]

Now, we choose $T_{2}$ sufficiently small such that 
\begin{equation*}
\left\{
\begin{aligned}
4n\Lambda t^{\beta} + \frac{8}{\beta} t b_{0}^{2}(t) + t c_{0} t &< \frac{(1 - \beta)}{2}, \\
2c_{0}(t) t^{\beta-1} &< \frac{\beta}{2},
\end{aligned}
\right.
\end{equation*}
holds for $0 < t \leq T_{2} \leq 1.$ Note that the above choice of $T_{2}$ is possible in view of the assumptions on $b_0$ and $c_0$. Finally, \eqref{req1} follows by taking $\gamma_{2} = \min \left\{\beta/2, \ (1 - \beta)/2\right\}$.
\end{proof}
\section{Main Result in the Case of the Base Boundary}

\begin{theorem}\label{bmain}
Assume that $b$ and $c$ satisfy the above-mentioned conditions. Let $E \subset \overline{\Omega}$ with 
\[ \mathscr{H}(E) < \frac{\lambda}{\Lambda} \]
and let $y_0$ be a point on $E$. Let $u$ be a real-valued function in $\Omega \times (0, T)$ of class $C^{2,1}$. Suppose that for some constant $L > 0$ and $\tau \in (0, \beta)$, we have:
\[ u(x,t) \geq -L \quad \text{and} \quad u \leq L t^{-\tau} \quad \text{for} \quad (x,t) \in \Omega_{T}. \]
Then $u \geq 0$ on $\partial \Omega_{T} \setminus (E \times \{0\})$ near $(y_0, 0)$ implies $u \geq 0$ at $(y_0, 0)$.
\end{theorem}

\begin{proof}
As the given PDE is translation invariant, without loss of generality we can assume that $y_0 = 0$. From the given hypothesis, we can find a ball $\mathscr{B}_{r} \subset \mathbb{R}^{n+1}$ such that
\[ u \geq 0 \quad \text{on} \quad \overline{\mathscr{B}_{r}} \cap [\partial \Omega_{T} \setminus (E \times \{0\})] \]
for all $0 < r < r_0$, for some $r_0 > 0$. Our aim is to produce an open neighborhood $\mathscr{N}$ of $(0, 0)$ in $\mathbb{R}^{n+1}$ such that $\liminf_{(x,t) \rightarrow (0, 0)} u \geq 0$ in $\mathscr{N} \cap \Omega_{T}$. Let us begin by observing that the function $u + \phi$ satisfies 
\[ \mathbb{L}(u + \phi) \leq \mathbb{L}(u) + \mathbb{L}(\phi) \leq (M t^{\beta - \tau} - \gamma_{2}) t^{-\beta}, \]
where $\phi$ is constructed in Lemma \ref{phi11}. Since $\beta > \tau$, if necessary, we may choose $r$ smaller such that:
\begin{equation}\label{req333}
\mathbb{L}(u + \phi) < 0 \quad \text{in} \quad \mathscr{B}_{r} \cap \Omega_{T}.
\end{equation}
Now, let us define $\delta = \left[\frac{\lambda}{\Lambda} - \mathscr{H}(E)\right] / 2$ and choose positive constants $\alpha$ and $\sigma$ satisfying
\[ 0 \leq \mathscr{H}(E) = \frac{\lambda}{\Lambda} - 2\delta < \frac{\lambda}{\Lambda} - \delta < 2\alpha < 4n\lambda\sigma < \frac{\lambda}{\Lambda}. \]

In view of $\frac{\lambda}{\Lambda} - \delta - 2\alpha < 0$, we can choose $\nu > 0$ sufficiently small satisfying 
\begin{equation}\label{123}
0 < \nu < r, \quad r + \nu^{2} < T \quad \text{and} \quad \gamma_{1} \nu^{\frac{\lambda}{\Lambda} - \delta - 2\alpha} > L.
\end{equation}
Moreover, by definition of $\delta$, we have $\mathscr{H}(E) < \frac{\lambda}{\Lambda} - \delta$, so for any $\epsilon > 0$, Lemma~\ref{house} provides an open covering $\{B_{i}\}$ of $E$ by balls in $\mathbb{R}^{n}$ with the following properties:
\begin{equation}\label{req111}
\left\{
\begin{aligned}
&\text{The centers of these balls lie in } E, \\
&\text{The radii of the balls satisfy } 0 < r_{i} \leq \nu, \\
&\sum_{i} r_{i}^{\frac{\lambda}{\Lambda} - \delta} < \epsilon.
\end{aligned}
\right.
\end{equation}
Corresponding to each ball $B_{i}$ in the covering, we define the following inverted paraboloids:
\[ P_{i} := \{(x,t) \in \mathbb{R}^{n+1} \mid |x - y_{i}|^{2} + t < r_{i}^{2} \}. \]
Observe that each paraboloid contains the corresponding ball from the covering. Therefore, $\{P_{i}\}$ also covers $E \times \{0\}$. Let us define open sets:
\[ A = \mathscr{B}_{r} \cap \Omega_{T} \quad \text{and} \quad P := \bigcup_{i} P_{i}. \]
Since $t \leq \nu < r$ in $P_{i}$ for each $i$, we find that $A \setminus \overline{P}$ is nonempty. Now, we show that $\mathscr{B}_{r}$ plays the role of $\mathscr{N}$. To show this, let us define the following auxiliary function:
\[ w(x,t) = u(x,t) + \left(1 + \frac{L}{r^{2}}\right) \phi(x,t) + \sum_{i} r_{i}^{\frac{\lambda}{\Lambda} - \delta} \psi(x - y_{i}, t + r_{i}^{2}). \]
Observe from \eqref{123} that $t + r_{i}^{2} \leq t + \nu^{2} \leq r + \nu^{2} < T$ in $\mathscr{B}_{r} \cap \Omega_{T}$; thus, each function $\psi(x - y_i, t + r_i^2)$ is defined for $(x, t)$ in $\mathscr{B}_{r} \cap \Omega_{T}$, and moreover will satisfy appropriately translated versions of \eqref{req11}. In fact, the estimates in \eqref{req11} and the sum in \eqref{req111}, along with the estimate $\psi(x,t) \leq t^{-\alpha}$, show that the series in the definition of $w$ converges uniformly on compact subsets of $\mathscr{B}_{r} \cap \Omega_{T}$, along with the series of space derivatives up to order two and the time derivative up to order one. Thus, the series represents a function in $\mathscr{B}_{r} \cap \Omega_{T}$ of class $C^{2,1}$, and the series may be differentiated termwise up to these orders. From \eqref{req1}, \eqref{req222}, and \eqref{req333}, we conclude therefore that in $\mathscr{B}_{r} \cap \Omega_{T}$,
\[ \mathscr{L} w \leq \mathscr{L} u + \left(1 + \frac{L}{r^2}\right) \mathscr{L} \psi \leq \mathscr{L}(u + \psi) < 0. \]
Now, by applying the minimum principle, we want to show that $w \geq 0$ in $A \setminus \overline{P}$.

\textbf{Claim:} This will be accomplished once we have 
\[ w \geq 0 \quad \text{on} \quad \partial_{\mathscr{P}} (A \setminus \overline{P}). \]

\textbf{Proof of the claim:} For this, we first note that 
\begin{align*}
\partial(A \cup (\overline{P})^{c}) &\subset \left(\partial A \cap (\overline{P})^{c}\right) \cup \left(A \cap \partial \overline{P}\right) \cup \left(\partial A \cap \partial \overline{P}\right) \\
&\subset \left(\partial A \setminus P\right) \cup \left(A \cap \partial P\right) \\
&\subset \left(\partial_{\mathscr{P}} (\mathscr{B}_{r}) \cap \overline{\Omega_{T}} \setminus P\right) \cup \left(\mathscr{B}_{r} \cap \partial_{\mathscr{P}} \Omega_{T} \setminus P\right) \cup \left(\mathscr{B}_{r} \cap \left(A \cap \partial P\right)\right).
\end{align*}

Based on the above containment of $\partial (A \setminus \overline{P})$, we consider three cases separately:

\textbf{Case I:} On $\partial (\mathscr{B}_{r}) \cap \overline{\Omega}_{T} \setminus P$, we have $0 \leq t \leq 1$, $|x|^{2} + t^{2} = r^{2}$. Therefore, we have 
\begin{align*}
w(x,t) &\geq u(x,t) + \left(\frac{L}{r^{2}}\right) \phi(x,t) \\
&\geq -L + \frac{L}{r^{2}} (t^{2} + |x|^{2}) = 0.
\end{align*}

\textbf{Case II:} On $\left(\mathscr{B}_{r} \cap \partial_{p} \Omega_{T} \setminus P\right)$, we have $w \geq 0$ since $u \geq 0$ on $\overline{\mathscr{B}_{r}} \cap \left[\partial (\Omega_{T}) \setminus (E \times \{0\})\right]$ by hypothesis and $\left(\mathscr{B}_{r} \cap \partial (\Omega_{T}) \setminus P\right) \subset \overline{\mathscr{B}_{r}} \cap \left[\partial (\Omega_{T}) \setminus (E \times \{0\})\right]$.

\textbf{Case III:} On $A \cap \partial P$, we have $t > 0$. Also, for $t \in P_{i}$, $t \leq r_{i}$ for each $i$ and $r_{i} \rightarrow 0$ as $i \rightarrow \infty$, in view of \eqref{req111}. So for each $(x,t) \in A \cap \partial P$, there exists a positive constant $d_{(x,t)}$ such that the distance from $(x,t)$ to $\overline{P}_{i}$ is larger than $d_{(x,t)}$ for all but finitely many $i$. Thus, for each $(x,t) \in A \cap \partial P$, there exists a positive integer $N$ such that
\[
(x,t) \in \partial \left(\bigcap_{i=1}^{N} P_{i}\right) \subset \bigcap_{i=1}^{N} \partial P_{i}. 
\]
Thus, $(x,t) \in \partial P_{i}$ for some $i$. Then we have
\begin{align*}
w(x,t) &\geq u(x,t) + r_{i}^{\frac{\lambda}{\Lambda} - \delta} \psi(x - y_{i}, t + r_{i}^{2}) \\
&\geq -L + r_{i}^{\frac{\lambda}{\Lambda} - \delta} \left(\gamma_{1} r_{i}^{-2\alpha}\right) \\
&\geq -L + \gamma_{1} \nu^{\frac{\lambda}{\Lambda} - \delta - 2\alpha} \geq 0.
\end{align*}

This completes the proof of the claim. Now, by the definition of $\psi$ and the sum in \eqref{req111}, the series in the definition of $w$ is bounded by $\frac{\epsilon}{t^{\alpha}}$. Therefore, $w \geq 0$ in $\mathscr{B}_{r} \cap \Omega_{T} \setminus \overline{P}$ implies
\[
u(x,t) + \left(1 + \frac{L}{r^{2}}\right) \phi(x,t) \geq -\epsilon t^{\alpha}.
\]
If we let $\nu \rightarrow 0^{+}$, while preserving \eqref{123}, and $\mathscr{B}_{r} \cap \Omega_{T} \setminus \overline{P}$ widens to include all of $\mathscr{B}_{r} \cap \Omega_{T}$ since $t \leq \nu$ in $P$; since $\epsilon$ is also arbitrary, we get
\[
u(x,t) + \left(1 + \frac{L}{r^{2}}\right) \phi(x,t) \geq 0,
\]
in $\mathscr{B}_{r} \cap \Omega_{T}$. As by the definition of $\phi$ in \eqref{phi}, we have $\phi(x,t) \rightarrow (0,0)$, we have $\phi \rightarrow 0$, and we get $\liminf_{(x,t) \rightarrow (0,0^{+})} u(x,t) \geq 0.$
\end{proof}
\section{Singular Barrier and Exceptional Lateral Boundary}

For a given $\theta \in (0, \pi)$ and $R > 0$, we denote by 
\[ C_{\theta,R} = B_{R}(0) \cap \left( \left\{ x \in \mathbb{R}^{N} \mid \arccos\left(\frac{x_{n}}{|x|}\right) \leq \theta \right\} \cup \{0\} \right). \]
the truncated cone with vertex at $0$, axis along $x_{n}$, aperture $\theta$, and radius $R$. The cone with an arbitrary vertex $x_{0} \in \mathbb{R}^{n}$ can be obtained by translating the cone with vertex at $0$, and we will denote it by $C_{\theta,R}(x_{0})$. Now, we present the following result from Theorem 2 \cite{miller1967barriers}, see also \cite{banerjee2015boundary,crandall1999existence}.

\begin{theorem}\label{baaa}
Given $\theta_{0} \in (0, \pi)$ and $\lambda < \Lambda$, there exist positive constants $\eta(\lambda, \Lambda, n, \theta_{0})$ and $\mu(\lambda, \Lambda, n, \theta_{0})$, and a positive function $h \in C^{2}(0, \theta_{0})$ such that for all $|\alpha| < \mu$, the function $v(x) = |x|^\alpha h(\theta(x))$ satisfies 
\begin{equation}
\mathcal{M}_{\lambda, \Lambda}^+(D^2v(x)) \leq -\eta |x|^{\alpha - 2},
\end{equation}
where $\theta(x) = \arccos\left(\frac{x_{n}}{|x|}\right)$ and $x \in \text{int}(C_{\theta_{0}, R})$.
\end{theorem}
There are some remarks:
\begin{itemize}
\item{} If $\alpha>0$ is positive in the above definition then the function $v$ is called regular barrier. While if $\alpha<0$ then the barrier is called singular.
\item{} At this juncture, we would also like to point out that the similar barrier also works for degenerate fully nonlinear elliptic operator. For the details see, section 3.1\cite{birindelli2009eigenvalue}.
\item{} We by choosing $R$ and $\mu$ further small in the above theorem we can show that similar statement also hold if we replace $\mathcal{M}_{\lambda, \Lambda}^+$ by more general elliptic operator with bounded lower order terms. For example $\mathscr{M}_{t}$ defined below. For the details in the case of linear elliptic equation we refer to Theorem 10.8.5, Lemma10.8.6 and Corollary 10.8.7\cite{helms2009potential}.
\item{} By examining the expression of $v$, we can also show that there are positive constants $C_{1}, C_{2},$... $C_{5}$, and an $r_{0} > 0$ such that all the conditions (i)-(v) of the Definition \ref{reff} hold if $z=0,$ that is $h_{0}(x)=v(x).$ Here we have assumed that $(0,t)$ lies on the parabolic boundary of the cylinder $\Omega_{T}.$ Moreover, as the class of equations considered here are invariant under rotation and translation so the similar barrier can be obtained by translating the vertex of the cone and corresponding function.
\end{itemize}
\subsection{Family of Elliptic Operators and associated barriers}
For each fixed $t \in (0, T)$, consider the following family of elliptic operators:
\[ \mathscr{M}_{t} = \mathcal{M}_{\lambda, \Lambda}^+(D^{2}(\cdot)) + \sum_{i=1}^{N} b_{i}(\cdot, t) \frac{\partial}{\partial x_{i}} + c(\cdot, t), \]
where $|b(x, t)| \leq K$ and $c(x, t) \leq 0$ for all $(x, t) \in \Omega \times (0, T)$. Let $E \subset \partial \Omega$ and $\{h_{z}\}$ be a family of functions indexed by points $z \in E$.

\begin{definition}\label{reff}
The family of functions $\{h_{z}\}$ is called a family of uniformly strong local barriers of order $\mu$ for the family of $\{\mathscr{M}_{t}\}$ in $(0, T)$ in $\Omega$ at $E$ if for some positive constants $r_{0}, C_{1}, C_{2}, C_{3}, C_{4}$, the functions $h(y, z):=h_{z}$ satisfy the following conditions:
\begin{enumerate}
\item Each $h_{z}$ is defined and continuous in $\overline{\Omega} \cap \{x \mid 0 < |x - y| < r_{0}\}$ and $C^{2}$ on $\Omega \cap \{x \mid 0 < |x - y| < r_{0}\}$.
\item For $z \in E$ and $y \in \Omega$ with $0 < |y - z| < r_{0}$,
\begin{equation}
C_{1} |y - z|^{\mu} \leq h(y, z) \leq C_{2} |y - z|^{\mu},
\end{equation}
\item $|D_{i} h(y, z)| \leq C_{3} |y - z|^{\mu - 1},$ for $i = 1, \ldots, n,$
\item $|D_{ij} h(y, z)| \leq C_{4} |y - z|^{\mu - 2},$
\item $\mathscr{M}_{t} h(y, z) \leq -C_{5} |y - z|^{\mu - 2},$ for $t \in (0, T).$
\end{enumerate}
\end{definition}

Now let us consider the operator $\mathscr{L} = -\frac{\partial u}{\partial t} + \mathcal{M}_{\lambda, \Lambda}^+(D^2u) + b(x,t) \cdot Du + c(x,t)u,$ where $|b| \leq K$ and $c \leq 0.$

\begin{theorem}\label{sinb}
Let $\mathscr{L}$ be the operator as given above in $\Omega_{T}$ and let $E \subset \partial \Omega$, with a point $z_{0} \in E.$ Suppose there exists a family of uniformly strong local singular barriers of order $-\mu$ for the family of elliptic operators \{$\mathscr{M}_{t}\}$ $t \in (0, T)$ in $\Omega$ at $E$ and there exists a strong local regular barrier of order $\eta$ for the same class of uniformly elliptic operators in $\Omega$ at $z_{0}.$ Suppose also that $\mathscr{H}(E) < \mu.$ Let $u \in C^{2,1}(\Omega_{T}).$ Suppose also that for some non-negative constant $M$ we have 
\begin{equation}\label{l}
u(x,t) \geq -L, \quad \mathscr{L} u(x,t) \leq L \quad \text{for} \quad (x,t) \in \Omega_{T}.
\end{equation}
Then for $0 < t_{0} < T,$ $u(z_{0}, t_{0}) \geq 0$ follows provided the boundary condition 
\begin{equation}\label{3.6}
u \geq 0 \quad \text{on} \quad (\partial \Omega \setminus E) \times (0, T) \quad \text{near} \quad (z_{0}, t_{0})
\end{equation}
is fulfilled.
\end{theorem}
\begin{remark}
The assumption regarding the existence of regular and singular barrier of order $\mu$ in the statement of above theorem is justified by the remark below Theorem \ref{baaa}.
\end{remark}
\begin{proof}
Let $h^{r} = h(\cdot, z_{0})$ and $h^{s} = h(\cdot, z_{0})$ be the regular and singular barriers of order $\eta$ and $-\mu$ at $z_{0} \in E$ in $\Omega$, respectively. In view of \eqref{3.6}, there exists an open cylinder $S$, say,
\[ S = \{(x,t) \in \mathbb{R}^{n+1} \mid |x - z_{0}| < r, \ |t - t_{0}| < s\} \]
such that 
\begin{equation}\label{later}
u \geq 0 \quad \text{on} \quad \left[ \left(\partial \Omega \setminus E\right) \times (0, T) \right] \cap \overline{S}.
\end{equation}

Define $\delta = [\mu - \mathscr{H}(E)] / 2$ and choose $\epsilon_{1}$ sufficiently small such that 
\begin{equation}\label{13}
0 < \epsilon_{1} < r \quad \text{and} \quad C_{1} \epsilon_{1}^{-\delta} \geq L.
\end{equation}

By the definition of the Hausdorff dimension, for given $\epsilon_{1} > 0$ and $\epsilon_{2} > 0$, we can find a covering of $E$ by balls $\{B_{r_{i}}(z_{i})\}$, where $z_{i} \in E$ and $0 < r_{i} \leq \epsilon_{1}$, such that
\begin{equation}\label{epsilon2}
\sum_{i} r_{i}^{\mu - \delta} < \epsilon_{2}.
\end{equation}

We may discard all the balls that do not intersect with $E \cap B_{r}(z_{0})$. Now for $(x,t) \in S \cap \Omega_{T}$ and each $i$,
\[ |x - z_{i}| \leq |x - z_{0}| + |z_{0} - z_{i}| \leq 2r + \epsilon_{1} < 3r < r_{0}. \]
Moreover, $|x - z_{0}| \leq r < r_{0}$. Therefore, $h^{r}(x)$ and all $h^{s}(x, z_{i})$ are defined and $C^{2,1}$ functions in $S \cap \Omega_{T}$. Furthermore, for any $(x,t) \in S \cap \Omega_{T}$, we also have 
\[ 0 < \text{dist}(x, E) \leq |x - z_{i}| < r_{0}. \]
Therefore, in view of \eqref{epsilon2} and conditions $(1)-(5)$ in the definition of the singular barrier, we find that the function   
\begin{equation}\label{w}
w(x,t) := u(x,t) + \left(1 + \frac{L}{C_{1} r^{\eta}}\right) h^{r}(x) + \sum_{i} r_{i}^{\mu - \delta} h^{s}(x, z_{i}) + \frac{L}{\eta^{2}} (t - t_{0})^{2}
\end{equation}
is well defined. Moreover, in view of \eqref{epsilon2}, we find that the termwise twice-differentiated series converges locally uniformly on $S \cap \Omega_{T}$. Therefore, this function is $C^{2,1}$ in $S \cap \Omega_{T}$. Furthermore, in view of \eqref{w}, \eqref{l}, and $(5)$ in the definition of the singular barrier, and the properties of the Pucci extremal operators, we have 
 \begin{equation*}
\begin{aligned}
\mathscr{L}(w) \leq \mathscr{L}(u)(x,t) + \mathscr{L}(g)(x) - \frac{2L}{s^{2}}(t - t_{0}) < 0.
\end{aligned}
\end{equation*}

Let us set $\hat{S} := S \cap \Omega_{T}$ and $P = \left(\cup_{i} B_{i}\right) \times \mathscr{R}$. Observe that $\hat{S} \setminus \overline{P}$ is nonempty provided $\epsilon_{1}$ is sufficiently small. In fact, any $(x,t) \in \hat{S}$ lies in $\hat{S} \setminus \overline{P}$ if $\epsilon_{1} < \text{dist}(x, E)$. To prove the claim, we first show that $w \geq 0$ on $\hat{S} \setminus \overline{P}$. Since $w$ is a supersolution, we need to show that $w \geq 0$ on the boundary of $\hat{S} \setminus \overline{P}$. Let us divide the boundary of $\hat{S} \setminus \overline{P}$ into the following three parts:
\[ \partial (\hat{S} \setminus \overline{P}) \subset (\partial S \cap \overline{\Omega_{T}} \setminus P) \cup \left(S \cap \partial \Omega_{T} \setminus P\right) \cup \left(\hat{S} \cap \partial P\right). \]
Let us treat the above three cases separately:
\begin{enumerate}
\item On $\partial S \cap \overline{\Omega_{T}} \setminus P$, we have either $|x - z_{0}| = r$ or $|t - t_{0}| = s$. If $|x - z_{0}| = r$, then by the definition of $w$, \eqref{l}, and condition $(1)$ in the definition of the barrier, and $|t - t_{0}| \leq s$, we have 
\[ w(x,t) \geq -L + \frac{L}{C_{1} r^{\eta}} C_{1} |x - z_{0}|^{\eta} = 0. \]
On the other hand, if $|t - t_{0}| = s$, again by the definition of $w$, \eqref{l}, and condition $(1)$ in the definition of the barrier, we have
\[ w(x,t) \geq -L + \frac{L}{s^{2}} s^{2} = 0. \]
\item Since $\hat{S} \cap \partial P \subset [\left(\partial \Omega \setminus E\right) \times (0, T)] \cap \overline{S}$, so $w \geq 0$ follows by \eqref{later}. 
\item Let $x \in \hat{S} \cap \partial P$, then we have $\text{dist}(x, E) > 0$. But the radius of the balls $B_{r_{i}}(z_{i})$ converges to zero by \eqref{epsilon2}, therefore $x \in \overline{B_{j}}$ for at most a finite number of $j$, say $N$. Therefore, we have 
\[ x \in \partial \left(\cup_{i=1}^{N} B_{i}\right) \subset \cup_{i=1}^{N} (\partial B_{i}), \]
which implies $x \in \partial B_{k}$ for some $k$. Then again by the definition of $w$, \eqref{l}, condition $(1)$, \eqref{13}, and $r_{j} \leq \epsilon_{1}$, we have 
\begin{align*}
w(x,t) &\geq u(x,t) + r_{k}^{\mu - \delta} h^{s}(x, z_{k}) \\
&\geq -L + r_{k}^{\mu - \delta} C_{1} |x - y_{k}|^{-\mu} \\
&\geq -L + C_{1} r_{k}^{-\delta} \geq -L + C_{1} \epsilon_{1}^{-\delta} \geq 0.
\end{align*}
\end{enumerate}

Thus, the minimum principle implies that $w \geq 0$ in $S \cap \Omega_{T} \setminus \overline{P}$. By condition $(1)$ in the definition of the barrier and \eqref{l}, the series defining $w$ can be bounded in $S \cap \Omega_{T}$ as follows:
\[ \sum_{i} r_{i}^{\mu - \delta} \leq \sum_{i} r_{i}^{\mu - \delta} C_{2} |x - z_{i}|^{-\mu} \leq \epsilon_{2} C_{2} [\text{dist}(x, E)]^{-\mu}; \]
Therefore, $w \geq 0$ implies 
\[ u(x,t) + \left(1 + \frac{L}{C_{1} s^{\eta}}\right) h^{r}(x) + \frac{L}{s^{2}} (t - t_{0})^{2} \geq -\epsilon_{2} C_{2} [\text{dist}(x, E)]^{-\mu}. \]
Now, we shrink $P$ by taking $\epsilon_{1} \rightarrow 0$ while preserving \eqref{13}. Since $\epsilon_{1}$ is arbitrary, we obtain that in $S \cap \Omega_{T}$, 
\[ u(x,t) + \left(1 + \frac{L}{C_{1} s^{\eta}}\right) h^{r}(x) + \frac{L}{s^{2}} (t - t_{0})^{2} \geq 0. \]
Now by using $\liminf_{z \rightarrow z_{0}} h^{r}(z) = 0$, we obtain $\liminf u \geq 0$ as $(x,t) \rightarrow (z_{0}, t_{0})$ in $S \cap \Omega_{T}.$
\end{proof}

By applying the above theorem to all the points of $E \times (0, T)$, we have the following corollary.

\begin{corollary}
If in Theorem \ref{sinb} the hypothesis \eqref{3.6} is replaced by  
\(u \geq 0 \quad \text{on} \quad (\partial \Omega \setminus E) \times (0, T) \quad \text{near} \quad E \times (0, T),\)
then we have $u \geq 0 ~ \text{on} ~ E \times (0, T).$
\end{corollary}

\bibliography{ref.bib}
\bibliographystyle{abbrv}
\end{document}